\documentclass{amsart}

\usepackage{graphicx}
\usepackage{pinlabel} 
\usepackage{amsmath}
\usepackage{amsfonts}
\usepackage{amssymb}
\usepackage{mathtools}
\usepackage{amscd}
\usepackage[all,cmtip]{xy}
\usepackage{amsthm}
\usepackage{comment}
\usepackage{anysize}
\usepackage{epsfig}
\usepackage{hyperref}

\newtheorem{theorem}{Theorem}[section]
\newtheorem{proposition}[theorem]{Proposition}
\newtheorem{corollary}[theorem]{Corollary}
\newtheorem{lemma}[theorem]{Lemma}

\newtheorem{conjecture}[theorem]{Conjecture}
  \theoremstyle{definition}
\newtheorem{definition}[theorem]{Definition}

  \theoremstyle{remark}
\newtheorem{remark}[theorem]{Remark}

\newcommand{\Kn}{K_n^{\operatorname{top}}}

\newcommand{\KR}{\mathbb{K}_R}

\newcommand{\dbC}{\mathbb{C}}
\newcommand{\dbE}{\mathbb{E}}

\newcommand{\dbH}{{\mathbb H}}
\newcommand{\dbK}{\mathbb{K}}

\newcommand{\dbQ}{\mathbb{Q}}
\newcommand{\dbR}{\mathbb{R}}

\newcommand{\dbT}{\mathbb{T}}
\newcommand{\dbZ}{\mathbb{Z}}

\newcommand{\calF}{{\mathcal F}}
\newcommand{\calH}{{\mathcal H}}
\newcommand{\calI}{{\mathcal I}}

\newcommand{\calM}{{\mathcal M}}
\newcommand{\calO}{{\mathcal O}}
\newcommand{\calP}{{\mathcal P}}

\newcommand{\philbert}{PSL_2(\mathcal{O}_k)}
\newcommand{\hilbert}{SL_2(\mathcal{O}_k)}

\newcommand{\vcyc}{V\text{\tiny{\textit{CYC}}}}
\newcommand{\fbc}{F\text{\tiny{\textit{BC}}}}
\newcommand{\inftyvcyc}{\vcyc_\infty }
\newcommand{\fin}{F\text{\tiny{\textit{IN}}}}
\newcommand{\all}{A\text{\tiny{\textit{LL}}}}

\newcommand{\LW}[4]{\xymatrix{ #1 \ar[r] \ar[d] &  #2 \ar[d] \\ #3 \ar[r] & #4 }}
\newcommand{\func}[3]{#1:#2\to#3}
\newcommand{\inv}[1]{ #1^{-1}}

\newcommand{\nbeq}{\begin{equation}}
\newcommand{\neeq}{\end{equation}}
\newcommand{\beq}{\begin{equation*}}
\newcommand{\eeq}{\end{equation*}}

\newcommand{\Nil}[1]{\bigoplus\limits_{i=0}^{#1}NK_{q-i}(R)^{\binom{#1}{i}}        }
\newcommand{\Nildos}[1]{\bigoplus\limits_{i=0}^{#1}NK_{q-i}(R(\dbZ/2))^{\binom{#1}{i}}        }

\newcommand{\dosNil}[1]{\bigoplus\limits_{i=0}^{#1}(NK_{q-i}(R)\oplus NK_{q-i}(R))^{\binom{#1}{i}}        }
\newcommand{\dosNildos}[1]{\bigoplus\limits_{i=0}^{#1}(NK_{q-i}(R(\dbZ/2))\oplus NK_{q-i}(R(\dbZ/2))^{\binom{#1}{i}}        }

\begin{document}

\title{The algebraic and topological K-theory of the Hilbert Modular Group}

\author{Luis Jorge S\'anchez Salda\~na}
\address{Unidad Cuernavaca del Instituto de Matem\'aticas, National University of Mexico, Mexico 62210}
\email{luisjorge@im.unam.mx}
\thanks{The first author was supported by FORDECYT postdoctoral grant and DGAPA-UNAM postdoctoral grant.}

\author{Mario Vel\'asquez}
\address{Departamento de Matem\'aticas, Pontificia Universidad Javeriana, Cra. 7 No. 43-82- Edificio Carlos Ort\'iz 5to piso, Bogot\'a D.C., Colombia}
\email{mario.velasquez@javeriana.edu.co}

\subjclass[2000]{Primary 54C40, 14E20; Secondary 46E25, 20C20}

\date{}


\keywords{K- and L- theory, Farrell-Jones conjecture, Baum-Connes conjecture, classifying spaces, equivariant homology theories}

\begin{abstract}
In this paper we provide descriptions of the Whitehead groups with coefficients in a ring of the Hilbert modular group and its reduced version, as well as for the topological K-theory of $C^*$-algebras, after tensoring with $\mathbb{Q}$, by computing the source of the assembly maps in the Farrell-Jones and the Baum-Connes conjecture respectively. We also construct a model for the classifying space of the Hilbert modular group for the family of virtually cyclic subgroups. \end{abstract}

\maketitle

\section{Introduction}

In \cite{BSS} Bustamante and S\'anchez studied the Whitehead groups of the Hilbert modular group $\hilbert$ and the reduced version $\philbert$, for any totally real extension $k$ of $\dbQ$. In that paper they obtained, for all $q\in \dbZ$, the following splitting 
\[ 
Wh_q(\philbert)\cong \bigoplus_{M\in\calF} Wh_q(M),
\]
where the sum runs over conjugacy classes of maximal finite subgroups of $\philbert$. Also they obtained the isomorphisms
\begin{align*}
Wh_1(\hilbert)&\cong \hilbert^{ab} \oplus \dbZ/2\oplus Wh_1(\philbert),\\
Wh_0(\hilbert) &\cong \dbZ \oplus Wh_0(\philbert), \text{ and }\\
Wh_{-1}(\hilbert) &\cong Wh_{-1}(\philbert).
\end{align*}
The main tool used in \cite{BSS} are the $K$-theoretic Farrell-Jones conjecture from \cite{FJ93}, the $p$-chain spectral sequence from \cite{DL03} and the action of the Hilbert modular group in the $n$-fold product of copies hyperbolic planes. It is worth noticing that they work with the Farrell-Jones conjecture with coefficients in the non-connective algebraic $K$-theory spectrum $\dbK_\dbZ$ of the integers. Nevertheless the computations carry on with coefficients in $\KR$, for any ring $R$, although they don't recover the classical Whitehead  groups, instead you get a direct summand of the Whitehead groups with coefficients (see Section 2).

In the present paper we are interested in the Whitehead groups with coefficients in an associative ring with unitary element, which are a generalization of the Whitehead groups studied in \cite{BSS}. The strategy is to use the $K$-theoretic Farrell-Jones conjecture to identify the Whitehead groups $Wh_q(G;R)$ with the homology groups of certain classifying spaces with coefficients in the nonconnective $K$-theory spectrum $\KR$. Then we use a result of  Bartels to split this homology theory into two parts. The first summand is the one studied in \cite{BSS}, while for the study of the second part we follow the strategy used in \cite{LR14}, which makes use of the inductive structure of our equivariant homology theory, the L\"uck-Weiermann construction from \cite{LW12}, and the Mayer-Vietoris sequence associated to it. The descriptions of the Whitehead groups of the Hilbert modular group and the reduced one are in Theorem \ref{whhilbert} and Theorem \ref{whphilbert}. The later it is a partial generalization of Theorem 3.35 in \cite{DKR11}.

On the other hand we also obtain a computation of the rational topological K-theory groups of the reduced group C$^\ast$-algebra of $\philbert$. We use the Baum-Connes conjecture to identify the topological K-theory groups with the equivariant K-homology groups of the classifying space for proper actions. Then we use the p-chain spectral sequence from \cite{DL03} and some results proved in \cite{BSS} that implies that $\philbert$ satisfies conditions (M) and (NM) defined in \cite{DL03}, finally using some computations of the rational group cohomology of $\philbert$ from \cite{Fr90} we obtain a complete calculation of the topological K-theory groups of the reduced C$^\ast$-algebra of the reduced Hilbert modular group in Theorem \ref{topk}.

This paper is organized as follows. In Section 2 we recall the definition of a classifying  space $E_{\calF}G$ of a group $G$ and a family of subgroups $\calF$, also we describe de L\"uck-Weiermann construction, and we construct an explicit model for the classifying space $E_{\fbc} V$ of a non-orientable virtually cyclic subgroups $V$ and the family of finite-by-cyclic subgroups. In Section 3 we recall the $K$-theoretic Farrell-Jones conjecture and the Baum-Connes conjecture, we introduce the Whitehead groups via a theorem of Waldhausen, and we recall Bartels splitting theorem of the domain of the assembly map in the Farrell-Jones conjecture. Section 4 is devoted to introduce the Hilbert modular group and what we call the \textit{ reduced} Hilbert group as well as some of their basic properties, we end this section by constructing models for the classifying space for both groups in  and the family of virtually cyclic subgroups in Theorem \ref{classifyinghilbert} and Theorem \ref{classifyinghilbert2}. Next, in Section 5 we perform the computation of the Whitehead groups of the Hilbert modular group and the reduced Hilbert modular group and give expressions in terms of the Whitehead groups of finite groups and the Nil-groups of the coefficient ring. Finally, in Section 6 we compute the rational topological $K$-theory of the reduced  $C^*$-algebra of the reduced modular Hilbert group using the Chern Character and some results in \cite{DL03}.


\section{Classifying spaces for families of subgroups}

In this section we recall the notion of classifying spaces for families of subgroups, the construction of L\"uck-Weiermann from \cite{LW12}, and the construction of a model for $E_{\fbc}V$, for a non-orientable virtually cyclic group $V$ and the family of finite-by-cyclic subgroups.\\

Let $G$ be a discrete group. A family of subgroups $\calF$ of a group $G$ is always assumed to be closed under conjugation and under taking subgroups. A model for the classifying space $E_\mathcal{F} G$  is a $G$-CW-complex $X$ satisfying that all of its isotropy groups belong to $\mathcal{F}$ and the fixed point set $X^H$ is contractible for every $H$ in $\calF$.  Equivalently, a model for $E_\mathcal{F} G$ is a terminal object in the category whose objects are $G$-CW-complexes with stabilizers in $\mathcal{F}$ (often called $\calF$-$G$-CW-complexes) and whose morphisms are $G$-homotopy classes of $G$-maps. It is well known that given $G$ and $\mathcal{F}$, a model for $E_\mathcal{F} G$ always exists and it is unique up to $G$-homotopy equivalence (see \cite{Lu05}). 

 We are specially interested in the
following families of subgroups:
\begin{itemize}
	\item $\all$ of all subgroups of $G$;
	\item $\vcyc$ of all virtually cyclic subgroups of $G$, i.e. subgroups which have a (possibly finite) cyclic subgroup of finite index;
    \item $\fbc$ of all subgroups that are either finite or isomorphic to $F\rtimes \dbZ$, with $F$ a finite group. Here the notation $\fbc$ stands for finite-by-cylic.
	\item $\fin$ of all finite subgroups;
    \item $Sub(K)$ the family of subgroups of $G$ generated by a subgroup $K\leq G$.
	\item $Tr$ consisting of the trivial subgroup.
\end{itemize}

The family $\fin$ is interesting because appears in the Baum-Connes conjecture. While the family $\vcyc$ appears in the original statement of the Farrell-Jones conjecture. In \cite{DKR11} and \cite{DQR11} it is proven that $\vcyc$ might be replaced by $\fbc$ in the Farrell-Jones assembly map, this is why this family it is also considered in this work.

We denote by $EG$, $\underline{E}G$, $\underline{\underline{E}}G$, and $E_KG$ the classifying spaces $E_{Tr}G$, $E_{\fin}G$, $E_{\vcyc}G$, $E_{Sub(K)}G$ respectively.

\begin{definition}
Let $G$ be a group and let $\inftyvcyc$ be the set of infinite virtually cyclic subgroups of $G$. 

\begin{itemize}
    \item Define an equivalence relation $\sim$ in $\inftyvcyc$ as follows: if $H,K\in\inftyvcyc$ we say that $H\sim K$ if $H\cap K$ is infinite. We denote by $[H]$ the equivalence class of $H$, and by $[\inftyvcyc]$ the quotient set.
    
    \item Define $N_G[H]:=\{g\in G| [g^{-1}Hg]=[H]\}$. We call this subgroup of $G$ the normalizer of the class $[H]$, or the commensurator of $H$. Note that $N_G[H]$ does not depend on the representant of the class $[H]$, in particular we can choose $H$ to be an infinite cyclic subgroup of $G$.
    
    \item Define a family of subgroups of $N_G[H]$ by $$\vcyc[H]:=\{V\subset N_G[H]| V\in \vcyc_\infty,V\in[H]\}\cup (\fin \cap N_G[H])$$
    where $\fin \cap N_G[H]$ consists of the family of finite subgroups of $N_G[H]$.
\end{itemize}
\end{definition}


\begin{theorem}\label{luckweiermannthm} \cite[Theorem 2.3]{LW12}
Let $G$ be a discrete group. Let $I$ be a complete set of representatives of the $G$-orbits in $[\inftyvcyc]$ under the $G$-action coming from conjugation. For every $H\in I$, choose models for $\underline{E}N_G[H]$ and $E_{\vcyc[H]}N_G[H]$, and a model for $\underline{E}G$. Now consider $X$ defined by the $G$-pushout:
	 $$ \xymatrix{ \coprod_{H\in I} G\times_{N_G[H]}\underline{E}N_G[H] \ar[r]^-i \ar[d]^{\coprod_{H\in I}Id_G\times_{N_G[H]}f_{H}} & \underline{E}G \ar[d] \\ \coprod_{H\in I}G\times_{N_G[H]} E_{\vcyc[H]}N_G[H] \ar[r] & X}$$
	 where the maps starting from the left upper corner are cellular and one of them is an inclusion of $G$-CW-complexes. Then $X$ is a model for $\underline{\underline{E}}G$.
\end{theorem}

Now we are going to analyze classifying spaces for virtually cyclic groups and the family $\fbc$ of finite-by-cyclic subgroups. It is well known that every virtually cyclic group $V$ is of one of the following types:
\begin{enumerate}
\item Finite;
\item orientable or finite-by-cyclic, i.e. it surjects onto $\dbZ$ with finite kernel, so that it is isomorphic to $F\rtimes \dbZ$ with $F$ a finite group; or
\item non-orientable, i.e. it surjects onto the infinite dihedral group $D_\infty$ with finite kernel, so that is isomorphic to an amalgam of finite groups $F_1*_{F_3} F_2$ with $[F_1:F_3]=[F_2:F_3]=2$.
\end{enumerate}

Let $V$ be a non-orientable virtually cyclic group, then we have the short exact sequence
\[
1\to F \to  V \xrightarrow{p_V} D_\infty \to 1,
\]
with $F$ a finite group. On the other hand, since $D_\infty$ is isomorphic to $\dbZ \rtimes \dbZ/2$, we have the short exact sequence
\[
1\to \dbZ \to D_\infty \xrightarrow{p_D} \dbZ/2 \to 1.
\]\\

The following theorem will be used in the proof of Lemma \ref{lemmaLR1}.

\begin{theorem}\label{EfbcV}
Let $V$ be a non-orientable virtually cyclic group and denote $P_V=p_D\circ p_V$. Let $X$ be the $V$-CW-complex defined by the $V$-pushout
\[
\LW{E_FV}{\underline{E}V}{P_V^*E\dbZ/2}{X,}
\]
where the upper arrow is an inclusion map, and $P_V^*$ denotes the action of $D_\infty$ induced by the map $P_V$. Then $X$ is a model for $E_{\fbc}V$.

\end{theorem}
\begin{proof}
Since for every $K\leq V$ we have the pushout
\[
\LW{E_FV^K}{\underline{E}V^K}{P_V^*E\dbZ/2^K}{X^K,}
\]
it is not difficult to see that $X$ is a model for $E_{\fbc} V$.
\end{proof}


\section{The Farrell-Jones conjecture and the Baum-Connes conjecture}

In this section we recall the Farrell-Jones conjecture and the Baum-Connes conjecture. Also we recall Bartel's splitting theorem and we introduce the Whitehead groups $Wh_q(G;R)$ using a famous theorem of Waldhausen.\\

Let $G$ be a discrete group and let $R$ be an associative ring with unitary element. We denote by $K_n(R(G))$, $n\in \dbZ$, the algebraic $K$-theory groups of the group ring $R(G)$ in the sense of Quillen for $n\geq0$ and in the sense of Bass for $n\leq -1$. Let $NK_n(R)$ denote the Bass Nil-groups of $R$, which by definition are the cokernel of the map in algebraic $K$-theory $K_n(R)\to K_n(R[t])$ induced by the canonical inclusion $R\to R(G)$. From the Bass-Heller-Swan theorem we get, for all $n\in\dbZ$, the decomposition $$K_n(R(\dbZ))\cong K_n(R[t,t^{-1}])\cong K_n(R)\oplus K_{n-1}(R) \oplus NK_n(R)\oplus NK_n(R).$$



Throughout this work we consider equivariant homology theories in the sense of \cite[Section 2.7.1]{LR05}. In particular, we are interested in the equivariant homology theory with coefficients in the $K$-theory spectrum described in \cite[Section 2.7.3]{LR05}, denoted by $H^G_*(-;\KR)$. For a fixed group $G$ this homology theory satisfies the Eilenberg-Steenrod axioms in the $G$-equivariant setting. One of the main properties of this homology theory is that 
\[  H_n^G(G/H;\KR) \cong H_n^H(H/H;\KR) \cong K_n(R(H))
\]
for every $H\subseteq G$. The other property we are interested in is the so-called \textit{induction structure} (see \cite[Section 2.7.1]{LR05}): given a group homomorphism $\alpha : H\to G$ and a $H$- CW- pair $(X,A)$ such that the kernel of $\alpha$ acts freely on $X$, there are, for every $n\in \dbZ$, natural isomorphisms $$ind_{\alpha} : H^H_n(X,A; \dbE )\to H^G_n(ind_{\alpha}(X,A);\KR).$$

This equivariant homology theory is relevant since it appears in the statement of the Farrell-Jones conjecture.\\

In their seminal paper \cite{FJ93} Farrell and Jones established their
famous isomorphism conjecture for the $K$-theory, 
$L$-theory and Pseudoisotopy functors. Here we consider the $K$-theoretic version of the conjecture as stated by Davis and L\"uck in \cite{DL98}.

\begin{conjecture}[The Farrell-Jones isomorphism conjecture] Let $G$ be group and let $R$ be a ring. Then, for any $n\in \dbZ$, the following assembly map, induced by the projection $\underline{\underline{E}}G\to G/G$, is an isomorphism

\nbeq\label{FJ}\tag{$\ast$}
\func{A_{\vcyc,\all}}{H^{G}_n(\underline{\underline{E}}G;\KR)}{H^{G}_n(G/G;\KR)\cong K_n(R(G))}.
\neeq

\end{conjecture}
Once the Farrell-Jones conjecture has been verified for a group $G$, one can
hope to compute $K_n(R(G))$ by computing the left hand side of \eqref{FJ}.
The later is a generalized homology theory that can be approached, for 
example, via Mayer-Vietoris sequences, Atiyah-Hirzebruch-type spectral
sequences or the $p$-chain spectral sequence described in \cite{DL03}.

In order to handle the left hand side of \eqref{FJ} is desirable to have good models for the classifying space $\underline{\underline{E}}G$. For this, we will use the construction of L\"uck and Weiermann described in the previous section. Roughly speaking, this construction, gives us an algorithm to construct a model for $\underline{\underline{E}}G$ using a model for $\underline{E}G$ and attaching some classifying spaces, respect to smaller families, of subgroups of $G$.\\

Given a cellular $G$-map $f:X\to Y$ between the $G$-CW-complexes $X$ and $Y$, we define $$H_n^G(f:X\to Y;\KR):=H_n^G(M_f, X;\KR),$$ where $M_f$ is the mapping cylinder of $f$ with the $G$-CW-structure induced by $X$ and $Y$, and $X$ is identified with the image of the canonical inclusion $X\to M_f$. Using the fact that $Y$ and $M_f$ are $G$-homotopy equivalent, we have the long exact sequence 
\begin{equation}\label{fexactseq}
\cdots \to H_n^G(X;\KR) \xrightarrow{f_*} H_n^G(Y;\KR) \to H_n^G(f:X\to Y;\KR)\to H_{n-1}^G(X;\KR) \to \cdots.
\end{equation}

If $g:X\to Y$ is a $G$-cellular map $G$-homotopic to $f$, then using a five lemma argument it is easy to see that $H_n^G(f:X\to Y;\KR)\cong H_n^G(g:X\to Y;\KR)$. In this context, the Farrell-Jones conjecture may be rephrased by claiming that $H^G_n(\underline{\underline{E}}G\to G/G;\KR)$ vanishes for every group $G$ and every $n\in \dbZ$.\\

The following theorem will be useful in our computations of algebraic $K$-theory.

\begin{theorem}\label{barthels}
 Let $G$ be a group and let $R$ be a ring. Then, for any $n\in \dbZ$, the assembly map induced by the (unique up to homotopy) $G$-map $\underline{E}G \to \underline{\underline{E}}G$
 $$H^G_n(\underline{E}G;\dbK_R)\to H^G_n(\underline{\underline{E}}G;\dbK_R)$$
 is split-injective, so that 
 \begin{equation}
 H^G_n(\underline{\underline{E}}G;\dbK_R)\cong H^G_n(\underline{E}G;\dbK_R) \oplus H^G_n(\underline{E}G\to \underline{\underline{E}}G;\KR).
 \end{equation}
\end{theorem}
\begin{proof}
The first assertion is the main result of \cite{Ba03}. In order to proof the splitting we use the long exact sequence (\ref{fexactseq}) to get, for any $n\in \dbZ$, the split short  exact sequence:
$$0\to H_n^G(\underline{E}G;\KR)\to H_n^G(\underline{\underline{E}}G;\KR)\to H_n^G(\underline{E}G\to \underline{\underline{E}}G;\KR)\to 0,$$ now the result follows.
\end{proof}

The Whitehead groups $Wh_n(G;R)$ of $G$ with coefficients in the ring $R$, appear in this context as follows.
\begin{proposition}\label{waldhausen}
\cite[Prop. 15.7]{Wa78} Let $G$ be a group. Then 
$Wh_n(G;R)\cong H_n^G(EG\to pt;\KR)$ for all $n\in\dbZ$.
In fact they fit in a long exact sequence

$$\cdots\to H^G_n(EG;\KR) \to K_n(R(G))\to Wh_n(G;R) \to
H_{n-1}^G(EG;\KR)\to \cdots.$$
\end{proposition}

\begin{lemma}\label{splitting}

Let $G$ be a group and let $R$ be a ring. Suppose that $G$ satisfies the Farrell-Jones conjecture. Then, for all $n\in \dbZ$, we have the following isomorphisms
\begin{align*}
  K_n(R(G))\cong H^G_n(\underline{\underline{E}}G;\dbK_R) &\cong H_n^G(\underline{E}G;\dbK_R)\oplus H_n^G(\underline{E}G\to \underline{\underline{E}}G;\dbK_R), \\
  Wh_n(G;R):=H_n^G(EG\to \underline{\underline{E}}G;\KR)  &\cong H_n^G(EG \to \underline{E}G;\KR)\oplus H_n^G(\underline{E}G \to \underline{\underline{E}}G;\KR).
\end{align*}

\end{lemma}
\begin{proof}
Consider the $G$-maps (unique up to $G$-homotopy) $f_1:EG\to \underline{E}G$, $f_2:\underline{E}G\to \underline{\underline{E}}G$ and $f_3:EG \to \underline{\underline{E}}G$. Up to $G$-homotopy we have that $f_3$ and $f_2\circ f_1$ are equal, hence we have that the mapping cylinder $M_{f_3}$ is $G$-homotopy equivalent to $M_{f_1}\cup_{\underline{E}G}M_{f_2}$. Now, from the long exact sequence  of the triple $(M_{f_3}, M_{f_1}, EG)$ and excision, we get the following long exact sequence
$$\cdots \to H_n^G(EG\to \underline{E}G;\KR)\to H_n^G(EG\to \underline{\underline{E}}G;\KR)\to H_n^G(\underline{E}G\to \underline{\underline{E}}G;\KR)\to \cdots, $$
which fits in the commutative diagram
$$\xymatrix{\cdots \ar[r] & H_n^G(EG\to \underline{E}G;\KR) \ar[r] & H_n^G(EG\to \underline{\underline{E}}G;\KR) \ar[r] & H_n^G(\underline{E}G\to \underline{\underline{E}}G;\KR)  \ar[r] & \cdots\\
\cdots \ar[r] & H_n^G(\underline{E}G;\KR) \ar[u] \ar[r] & H_n^G(\underline{\underline{E}}G;\KR) \ar[u]\ar[r] & H_n^G(\underline{E}G \to \underline{\underline{E}}G;\KR) \ar[u]^{=}\ar[r] & \cdots,
}$$
where the vertical arrows are the induced by the inclusions $(\underline{E}G,\emptyset)\hookrightarrow (M_{f_1}, EG)$ and $(\underline{\underline{E}}G,\emptyset)\hookrightarrow (M_{f_3}, EG)$. From Theorem \ref{barthels} we know that the lower row splits, which leads to a splitting of the upper row and this finishes the proof.
\end{proof}
Using this lemma we can divide our task into two parts: 

\begin{itemize}
    \item The computation of $H_n^G(EG\to\underline{E}G;\dbK_R)$, which in our case it is possible using the p-chain spectral sequence following the strategy as in \cite{BSS}, and
    \item the computation of $H_n^G(\underline{E}G\to \underline{\underline{E}}G;\dbK_R)$, which can be done, following the strategy in \cite{LR14}, using the L\"uck-Wiermann construction, the induction structure of the equivariant homology theory $H_*^G(-;\KR)$ and the explicit computation of the commensurators $N_G[H]$. 
\end{itemize}

 Now we proceed to describe the Baum-Connes conjecture, consider the complex group ring $\dbC(G)$ cannonically included in $L^2(G)$, the \emph{reduced group C$^\ast$-algebra of $G$} is the completion of $\dbC(G)$ in $L^2(G)$ and is denoted by $C_r^*(G)$. For $n\in\dbZ$, let $\Kn(C_r^*(G))$ denotes the topological K-theory groups of $C_r^*(G)$.

The Baum-Connes conjecture relates the topological K-theory groups of $C_r^*(G)$ with a more accesible object, namely the so called  \emph{$G$-equivariant K-homology}, it is the equivariant homology theory with coefficients in the topological K-theory spectrum $\dbK^{\operatorname{top}}$ described for example in \cite[Section 2]{LR05}. This theory is denoted by $H_n^G(-;\dbK^{\operatorname{top}})$. Due to the induction structure we have a very important property of this theory
 $$H_n^G(G/H;\dbK^{\operatorname{top}})\cong H_n^H(H/H;\dbK^{\operatorname{top}}) \cong \Kn(C_r^*(H)).$$

 \begin{conjecture}[Baum-Connes conjecture]Let $G$ be a discrete group. Then for any $n\in\dbZ$, the following assembly map, induced by the projection $\underline{E}G\to G/G$, is an isomorphism
 \nbeq\label{BC}\tag{$\ast\ast$}
\func{A_{\fin,\all}}{H^{G}_n(\underline{E}G;\dbK^{\operatorname{top}})}{H^{G}_n(G/G;\dbK^{\operatorname{top}})\cong \Kn(C_r^*(G))}.
 \neeq
 \end{conjecture}


\section{The Hilbert modular group}

In this section we review the definitions of the Hilbert modular group and its reduced version, next we recall some of their basic  properties. Then, we proceed to construct models for the classifying spaces for the family of virtually cyclic subgroups.
The results we state without proof in this section can be found
for example in \cite{Fr90}. For additional information about the Hilbert modular group we refer the reader to \cite{Hi73}, \cite{Ge88}, and to \cite{Ef87} for the reduced Hilbert group.\\

A \textit{totally real number field} $k$ is a finite extension of $\dbQ$, of degree $n$, such that all its embedings $\sigma_i:k\to \dbC$ have image contained in $\dbR$.

\begin{definition}
 Let $k$ denote a totally real number field of degree $n$ and $\calO_k$ its ring of algebraic integers. The \textit{Hilbert modular group} is by definition $SL_2(\calO_k)$, also we call the quotient $PSL_2(\calO_k)=SL_2(\calO_k)/\{I,-I\}$ the \textit{reduced Hilbert modular group}, where $I$ is the identity matrix. From now on we will denote the Hilbert modular group with the letter $\Gamma$, the reduced Hilbert modular group with the letter $G$, and $p:\Gamma \to G$ the canonical projection.\\
\end{definition}

By definition an element in $\philbert$ is a class that has exactly two representants, say a matrix $A$ and its negative $-A$. From now on, we will make an abuse of notation and we will not distinguish between a matrix in $\hilbert$ and its class in $\philbert$.\\

Note that if $k=\dbQ$, then $PSL_2(\calO_k)=PSL_2(\dbZ)$ is nothing but the classical
modular group, hence it is a discrete subgroup of $PSL_2(\dbR)$, and admits a proper and discontinuous action in the hyperbolic plane $\dbH$ via M\"obius transformations. However, $PSL_2(\calO_k)$ \textit{is not} a discrete subgroup
of $PSL_2(\dbR)$ if $n\geq 2$, where $n$ is the degree of $k$ over $\dbQ$. Yet we can define the embedding
\begin{equation}\label{hilbertembedding}
\sigma: PSL_2(\calO_k)\to PSL_2(\dbR)^n=PSL_2(\dbR)\times\cdots\times PSL_2(\dbR)\end{equation}
by
\begin{equation*}
 \begin{pmatrix}
   \alpha & \beta\\
   \gamma & \delta
 \end{pmatrix} \mapsto \left(  \begin{pmatrix} 
   \sigma_1(\alpha) & \sigma_1(\beta)\\
   \sigma_1(\gamma) & \sigma_1(\delta)
 \end{pmatrix},...,  \begin{pmatrix}
   \sigma_n(\alpha) & \sigma_n(\beta)\\
   \sigma_n(\gamma) & \sigma_n(\delta)
 \end{pmatrix}  \right),   
\end{equation*}
and identifying the Hilbert modular group with its image we can think of it as a discrete subgroup of $PSL_2(\dbR)^n$. Considering the diagonal action of $PSL_2(\dbR)^n$ in the $n$-fold product $\dbH\times\cdots\times\dbH$, we have that the Hilbert modular group it does act properly and discontinuously in this $n$-fold product space. All we have established above is also true 
for $\hilbert$ in the non-projectivised setting.\\



Next, we would like to analyze the commensurators $N_G[H]$, for $G=\philbert$, the reduced Hilbert modular group and $H$ an infinite virtually cyclic subgroup. Since the commensurator does not depend on the commensuration class of $H$ we can asume that $H$ is an infinite cyclic group generated by an infinite order element $\alpha \in G$. Because of this it is useful to have a classification of elements in the reduced Hilbert modular group, so that we can analyze the commensurators case by case. Recall that we are considering $\dbH$ with the model of the upper half plane. 

\begin{definition}
  Consider an element $\alpha\in PSL_2(\dbR)$. We say that $\alpha$ is 
  \begin{itemize}
      \item elliptic if $Tr(\alpha)^2<4$,
      \item parabolic if $Tr(\alpha)^2=4$, and
      \item hyperbolic if $Tr(\alpha)^2>4$.
    \end{itemize}
\end{definition}

Note that the action of $PSL_2(\dbR)$ on $\dbH$ can be extended to an action of $\overline{\dbH}=\dbH\cup \dbR \cup\{\infty\}$.

\begin{lemma}
Consider an element $\alpha\in PSL_2(\dbR)$. Then
  \begin{enumerate}
     \item $\alpha$ is elliptic if and only if it fixes exactly one point in $\dbH$.
     \item $\alpha$ is parabolic if and only if it fixes exactly one point in $\dbR \cup\{\infty\}$.
     \item $\alpha$ is hyperbolic if and only if it fixes exactly two points in $\dbR \cup\{\infty\}$.
 \end{enumerate}
\end{lemma}

It may happen that the image of the matrix $\begin{pmatrix}
   \alpha & \beta\\
   \gamma & \delta
 \end{pmatrix}$ in the reduced Hilbert modular group $PSL_2(\calO_k)$ under the embedding (\ref{hilbertembedding}) has both elliptic components and hyperbolic components. For example, if $k=\dbQ(\sqrt{2})$, then there are two embeddings of $k$ into $\dbR$, namely
\begin{equation}
\sigma_1: s+t\sqrt{2}\mapsto s+t\sqrt{2}\ \ \ \ \ \ \text{and}\ \ \ \ \ \ \sigma_2:s+t\sqrt{d}\mapsto s-t\sqrt{2},
\end{equation}
for $s,t\in\dbQ$. In this case, the matrix  
$\begin{pmatrix}
1+\sqrt{2} & 1+\sqrt{2}\\
2 & 1+\sqrt{2}
\end{pmatrix}\in PSL_2(O_k)$ is mapped by the embedding (\ref{hilbertembedding}) to

\begin{equation*}
  \left(\begin{pmatrix}
1+\sqrt{2} & 1+\sqrt{2}\\
2 & 1+\sqrt{2}
\end{pmatrix}, \begin{pmatrix}
1-\sqrt{2} & 1-\sqrt{2}\\
2 & 1-\sqrt{2}
\end{pmatrix}\right).
\end{equation*}
which clearly has first component elliptic and second component hyperbolic. This is an example of what we will call a \textit{mixed} element. It is pointed out in \cite[Page 11, third paragraph]{Ef87} that cannot exist mixed elements with parabolic components.


\begin{definition}
  Consider an element $\alpha$ of the reduced Hilbert modular group $PSL_2(\calO_k)$, and denote by $\overline{\alpha}$ its image under (\ref{hilbertembedding}). 
  
  \begin{enumerate}
      \item We say that $\alpha$ is totally elliptic (resp. hyperbolic, parabolic) if all of the components of $\overline{\alpha}$ are elliptic (resp. hyperbolic, parabolic).
      \item We say that $\alpha$ is mixed if $\overline{\alpha}$ has both elliptic components and hyperbolic components.
      \item We say that $\alpha$ is hyperbolic-parabolic if it is totally hyperbolic and there exist one point in $(\dbR \cup \{\infty\})^n$ fixed by $\overline{\alpha}$ that it is fixed by some totally parabolic element of $PSL_2(\calO_k)$.      
  \end{enumerate}  
\end{definition}

Continuing with the example above, we can get an hyperbolic-parabolic element by considering $\alpha=\begin{pmatrix}
1+\sqrt{2} & 0\\
0 & (1+\sqrt{2})^{-1}
\end{pmatrix}$, then the fixed points of $\overline{\alpha}$ are $(\infty, \infty)$, $(0,0)$, $(\infty,0)$, and $(0,\infty)$ and  the first two points are fixed by a totally parabolic element.\\

Now, since it is clear that the elements of $PSL_2(\calO_k)$ of finite order are those that happen to be totally elliptic, Proposition \ref{mixedcom} and Proposition \ref{parabcom} describe the commensurators of all infinite cyclic subgroups of $\philbert$.

\begin{lemma}\label{normalizers}
Let $G=PSL_2(\calO_k)$. Consider an infinite order element $\alpha\in G$ and denote by $H$ the infinite cyclic subgroup generated by $\alpha$. Then, the normalizer $N_G(H)$ fits in the short exact sequence $$1\to C_G(H) \to N_G(H) \to F \to 1,$$ where $F$ is a subgroup of $\dbZ/2$, and $C_G(H)$ is the centralizer of $H$ in $G$. 
\end{lemma}
\begin{proof}
 Consider the action of $G$ in $\overline{\dbH}$. It is well known that the normalizer $N_G(H)$ acts in the fixed point set $\overline{\dbH}^H$ of $H$, which is exactly the fixed point set of $\alpha$. On the other hand, from hyperbolic geometry, we know that $\overline{\dbH}^H$ consists of at most two points (in the boundary of $\overline{\dbH}$). Hence, we have the short exact sequence $$1\to K \to N_G(H) \to F \to 1$$ where $F$ is a subgroup of $\dbZ/2$, and $K$ consists of the elements $g$ of $G$ such that they act trivially in $\overline{\dbH}^H$, i.e., $g$ has $\overline{\dbH}^H$ as its fixed point set . Again, by hyperbolic geometry we know that two elements in $G$ commute if and only if they have the same fixed point set, hence $K=C_G(H)$.
\end{proof}


\begin{proposition}\label{mixedcom}
Let $G=PSL_2(\calO_k)$. Consider an element $\alpha\in G$ with $m$ hyperbolic components, $1\leq m \leq n$, and let $H$ be the infinite cyclic subgroup generated by $\alpha$. Then,
$$N_G[H]\cong N_G(H)\cong  \begin{cases}
 \dbZ^m \text{ or } \dbZ^m\rtimes \dbZ /2         &\text{if $\alpha$ if not hyperbolic-parabolic, and}\\
 \dbZ^{n-1} \text{ or } \dbZ^{n-1}\rtimes \dbZ /2 &\text{if $\alpha$ is hyperbolic-parabolic,}
\end{cases}$$
where the action of $\dbZ/2$ in the semidirect products is given by multiplication by $-1$, in particular is free away from the origin.
\end{proposition}
\begin{proof}
First, we compute the normalizer $N_G(H)$. By Lemma \ref{normalizers}, we have two cases. If $F$ is trivial we have that $N_G(H)= C_G(H)$. Now, if $F= \dbZ/2$, then any element $\beta$ in $N_G(H)$ that is not in $C_G(H)$ acts non-trivially in the fixed point set of $\alpha$ (considering the action of $\alpha$ in $\overline{\dbH}$), hence $\beta$ acts on the geodesic that joins the fixed points of $\alpha$ by an involution, and we conclude that $\beta$ is an element of order two, so that the short exact sequence in Lemma \ref{normalizers} splits, and $N_G(H)\cong C_G(H)\rtimes \dbZ/2$.\\
Now we shall consider the action of $G$ in the $n$-fold product $\dbH^n$. Since the hyperbolic plane is a $CAT(0)$-space with the hyperbolic metric, we have that $\dbH^n$ is also a $CAT(0)$-space with the product metric. Moreover, $\alpha$ is a hyperbolic isometry (in the $CAT(0)$ sense), and we have that (see \cite[Proof of Proposition 4.4]{DP152} $$N_G[H]=\{g\in G| \exists n\in\dbZ \text{ such that }g^{-1}\alpha^ng=\alpha^{\pm n}\}=\bigcup\limits_{i=1}^{\infty} N_G(\langle \alpha^i \rangle).$$ 
Note that $\alpha^i$ and $\alpha$ have the same fixed point set, then $N_G(\langle \alpha^i \rangle)$ acts in the fixed point set $\overline{\dbH}^H$ of $\alpha$. We deduce that $N_G[H]$ acts in $\overline{\dbH}^H$, and we have the short exact sequence $$1\to \overline{K}\to N_G[H] \to F \to 1,$$ where $F$ is a subgroup of $\dbZ/2$, and $\overline{K}$ consists of those elemets of $N_G[H]$ that have $\overline{\dbH}^H$ as fixed point set, so that $\overline{K}=C_G(H)$. Finally, we have that $N_G[H]\cong C_G(H)\rtimes F$. If $F=\dbZ/2$, by restricting the action of $N_G[H]$ to the geodesic with end points in $\overline{\dbH}^H$, we can see that the  non-trivial element of $F$ acts by multiplication by $-1$.\\
The only thing left is to describe the centralizers, this was done by Efrat in Theorem 5.7 page 26 and Proposition 3.1 page 93 of \cite{Ef87}, where it is proved that $C_G(H)$ is a free abelian group of rank $m$ (resp. $n-1$) if $\alpha$ is not an hyerbolic-parabolic (resp. hyperbolic-parabolic) element.
\end{proof}

Considering, again, the hyperbolic-parabolic element $\alpha=\begin{pmatrix}
1+\sqrt{2} & 0\\
0 & (1+\sqrt{2})^{-1}
\end{pmatrix}$, we can see that the order two element $\beta=\begin{pmatrix} 0 & -1\\
1 & 0
\end{pmatrix}$ belongs to the normalizer $N_G(H)$, therefore we have the isomorphism $N_G(H)\cong \dbZ^{n-1}\rtimes \dbZ /2$. We still don't know whether there exist or not some infinite cyclic subgroup $H$ of the reduced Hilbert modular group for which the normalizer $N_G(H)$ is free abelian. 


\begin{proposition}\label{parabcom}
Let $G=PSL_2(\calO_k)$. Consider a totally parabolic element $\alpha\in G$, and denote by $H$ the infinite cyclic subgroup generated by $\alpha$. Then, we have that
    $$N_G[H]\cong N_G(H) \cong \dbZ^{n}.$$
\end{proposition}
\begin{proof}
First, we will suppose that $\alpha$ acting on the hyperbolic plane $\dbH$  fixes $\infty$. In this case $\alpha$ is a translation, i.e., is represented by a matrix of the following form $A=\begin{pmatrix}
1 & t\\
0 & 1
\end{pmatrix}$ with $t\in \calO_k$. Now suppose that the matrix $B=\begin{pmatrix}
a & b\\
c & d
\end{pmatrix}$ belongs to $N_G[H]$. Since $(B^{-1}AB)^n$ is a translation for some $n\in \dbZ$, hence $B^{-1}AB$ should be itself a translation and  a direct computation shows that $c=0$, $d=a^{-1}$, so that
$$B^{-1}AB=\begin{pmatrix}
1 & ta^{-2}\\
0 &1
\end{pmatrix}.$$
By definition of $N_G[H]$ we know that there exist $n,m\in \dbZ$ such that $A^m=B^{-1}A^nB$, which is equivalent to the following identity
$$\begin{pmatrix}
1 & mt\\
0 & 1
\end{pmatrix}=\begin{pmatrix}
1 & ta^{-2}n\\
0 & 1
\end{pmatrix},$$
it follows that $a^{-2}=\frac{m}{n}$, so that $a^{-2}\in \dbQ\cap \calO_k =\dbZ$, hence $a=\pm 1$. Since $\calO_k \cong \dbZ^n$ as abelian groups, we have that $N_G[H]\cong \dbZ^n$. We have shown that $B=\begin{pmatrix}
\pm 1 & b\\
0 & \pm 1
\end{pmatrix}$, this is, $B$ is a translation. It is now clear that $N_G[H]=N_G(H)$.

Now suppose that $\alpha$ does not fix $\infty$ and it is represented by the matrix $A'$. We can find a matrix $M\in GL_2(k)$ such that $M^{-1}A'M$ fixes infinity (see \cite[Proof of Proposition 3.4]{Fr90}). Using the argument in the previous case we can show that $M^{-1}N_G[H]M$ is an additive subgroup of $\dbR$. On the other hand $M^{-1} GM$ is commensurable with $G$ in $PSL_2(\dbR)$ (see \cite[Corollary 3.3]{Fr90}, hence $M^{-1}N_G[H]M$ and $N_G[H]$ are commensurable. We can conclude that $M^{-1}N_G[H]M$ is isomorphic to $\dbZ^n$. As is the previous case we can see that the elements in $M^{-1}N_G[H]M$ are translations, this implies that $N_G[H]=N_G(H)$.
\end{proof}


In order to construct a model for $\underline{E}G$ we will simplify the push-out given in Theorem \ref{luckweiermannthm}. For this we will need the following lemma.

\begin{lemma}\label{HinHmax}
 Let $G=PSL_2(\calO_k)$. Then each infinite cyclic subgroup $H$ of $G$ is contained in a unique maximal infinite cyclic subgroup $H_{max}$ of $G$.
\end{lemma}
\begin{proof}
Consider $H$ and $H'$ two infinite cyclic subgroup of $G$ such that $H\subseteq H'$. Then, it is clear that $H'$ is contained in $C_G(H)$ the centralizer of $H$ in $G$, which is a finitely generated free abelian group. Since every cyclic group contained in a finitely generated abelian group is contained in a unique maximal cyclic group, we have that $H$ is contained in a unique maximal infinite cyclic subgroup of $G$.
\end{proof}

\begin{theorem}\label{classifyinghilbert}
Let $G=PSL_2(\calO_k)$. Let $I$ be a complete set of representatives of conjugacy classes of maximal infinite cyclic subgroups of $G$. For every $H\in I$, choose models for $\underline{E}N_G(H)$ and $\underline{E}W_G(H)$, where $W_G(H)=N_G(H)/H$. Now consider the $G$-pushout:
	  $$ \xymatrix{ \coprod_{H\in I} G\times_{N_G(H)}\underline{E}N_G(H) \ar[r]^-i \ar[d]^{\coprod_{H\in I}Id_G\times_{N_G(H)}f_{H}} & \underline{E}G \ar[d] \\ \coprod_{H\in I}G\times_{N_G(H)} \underline{E}W_G(H) \ar[r] & X}$$
	  where $\underline{E}W_G(H)$ is viewed as an $N_G(H)$-CW-complex by restricting with the projection $N_G(H)\to W_G(H)$, the maps starting from the left upper corner are cellular and one of them is an inclusion of $G$-CW-complexes. Then $X$ is a model for $\underline{\underline{E}}G$.
\end{theorem}

\begin{proof}
We will verify that all elements in Theorem \ref{luckweiermannthm} can be replaced by those appearing in the above push-out. In fact, since each infinite virtually cyclic subgroup of $G$ is commensurable with an infinite virtually cyclic subgroup of $G$, using the Lemma \ref{HinHmax} we conclude that the set of representatives $I$ in Theorem \ref{luckweiermannthm} coincides with the $I$ defined in the statement above. Now, if we take $H\in I$, from Propositions \ref{mixedcom} and \ref{parabcom}  we have that $N_G[H_{max}]=N_G(H_{max})$. The only thing left is to see that a model for $E_{\vcyc[H]}N_G(H)$ is also a model for $\underline{E} W_G(H)$. We know that $N_G(H)$ is isomorphic either to $\dbZ ^r$ or to $\dbZ^r\rtimes \dbZ/2$, since $H$ is maximal, we have that $W_G(H)$ is either isomorphic to $\dbZ ^{r-1}$ or to $\dbZ^{r-1}\rtimes \dbZ/2$ respectively. Anyway, checking the assertion in both cases is straightforward.
\end{proof}

Recall that the geometric dimension of a group $G$ respect to the family of virtually cyclic subgroups $\underline{\underline{gd}}(G)$ is defined as the minimum $n$ such that there exists an $n$-dimensional model for $\underline{\underline{E}}G$.

\begin{corollary}
Let $\Gamma=\hilbert$, and $G=\philbert$. Then $\underline{\underline{gd}}(G)=\underline{\underline{gd}}(\Gamma)\leq 2n$.
\end{corollary}
\begin{proof}
Since Since $\Gamma$ is an extension of $G$ by a finite group it is clear that  $p^*\underline{\underline{E}}G=\underline{\underline{E}}\Gamma$, therefore $\underline{\underline{gd}}(G)=\underline{\underline{gd}}(\Gamma)$.
A model for $\underline{E}G$ is given by the $n$-fold product $\dbH\times \cdots \times \dbH$ (see \cite[Remark 4.2]{BSS}), while models for $\underline{E}N_G(H)$ and $\underline{E}W_G(H)$ are given by $\dbR^n$. Now the result follows from the pushout in the previous theorem.
\end{proof}

\begin{corollary}
Let $G=\philbert$ and let $R$ be a ring. Then, for every $q\in \dbZ$,
\begin{align*}
    H_q^G(\underline{E}G\to \underline{\underline{E}}G;\KR) &\cong H_q^G( \coprod_{H\in I} G\times_{N_G(H)}\underline{E}N_G(H) \to \coprod_{H\in I} G\times_{W_G(H)}\underline{E}W_G(H);\KR)\\
                &\cong \bigoplus\limits_{H\in I}H_q^{N_G(H)}(\underline{E}N_G(H)\to \underline{E}W_G(H);\KR)
\end{align*}
where $I$ is a set of representatives of conjugacy classes of maximal infinite cyclic subgroups.
\end{corollary}
\begin{proof}
   Using the $G$-pushout from the previous theorem, we have the first isomorphism. In fact, since we are assuming  that the map $i$ is an inclusion of a subcomplex, the space $X$ is actually the homotopy pushout, and then the homotopy cofibers of the two vertical rows are $G$-homotopically equivalent. Now, the second isomorphism follows from the induction structure of the equivariant homology theory, because 
 we have $H_q^G(G\times_{N_G(H)}\underline{E}N_G(H);\KR)\cong H_q^{N_G(H)}(\underline{E}N_G(H);\KR)$ and $H_q^G(G\times_{N_G(H)}\underline{E}W_G(H);\KR)\cong H_q^{N_G(H)}(\underline{E}W_G(H);\KR)$, and a five lemma argument.
\end{proof}


Once we have analyzed the commensurators of infinite cyclic subgroups of $G=\philbert$ and the classifying space $\underline{\underline{E}}G$, we are going to do so for $\Gamma=\hilbert$. For this we recall that we are denoting by $p:\Gamma \to G$ the canonical projection. Then, we have the following short exact sequence $$1\to Z\to \Gamma \xrightarrow{p} G\to 1,$$ where $Z=\left\{\pm\begin{pmatrix}
1 & 0\\
0 & 1
\end{pmatrix}\right\} \cong \dbZ/2$ is the center of $\Gamma$. We are going to use this notation for the rest of the paper. Note that we have a classification of elements in $\Gamma$ analogue to that in $G$, i.e. the elements in $\Gamma$ are either elliptic, parabolic or hyperbolic. 

\begin{lemma}\label{lemmacentranormGamma}
Consider $A\in \Gamma=\hilbert$ a hyperbolic or parabolic element. If $B$ is such that $p(B)p(A)p(B^{-1})=p(A^{\pm1})$, then $B A B^{-1}=A^{\pm1}$.
\end{lemma}
\begin{proof}
Since $p$ is a homomorphism of groups we have that $p(BA B^{-1})=p(A^{\pm1})$, hence $B AB^{-1}=\pm A^{\pm1}$. On the other hand, taking traces in the later equation we get $Tr(A)=Tr(B A B^{-1})=\pm Tr(A^{\pm1})=\pm Tr(A)$. If $B A B^{-1}=- A^{\pm1}$ we conclude that $Tr(A)=0$, contradicting the fact that $A$ is hyperbolic or parabolic. Now, the result follows.
\end{proof}

As an immediate consequence we have the following.

\begin{proposition}\label{normalizersl}
Consider $A \in \Gamma=\hilbert$ a hyperbolic or parabolic element, and let $C$ be the infinite cyclic subgroup generated by $A$. Then
\begin{enumerate}
\item $C_\Gamma(C) = p^{-1}(C_G(p(C)))$, and
\item $N_\Gamma (C) = p^{-1} (N_G (p(C)))$.
\end{enumerate}
\end{proposition}

\begin{lemma}\label{centralizergamma}
Consider $A \in \Gamma=\hilbert$ a hyperbolic or parabolic element, and let $C$ be the infinite cyclic subgroup generated by $A$. Then $C_\Gamma(C)\cong \dbZ/2 \oplus C_G(p(C))$.
\end{lemma}
\begin{proof}
From the previous proposition we have the short exact sequence $$1\to Z\to C_\Gamma(C) \to C_G(p(C)) \to 1.$$
Since $C_\Gamma(C)$ is abelian (see \cite[Lemma 2.2 p. 83]{Fr90}) and $C_G(C)$ is free abelian we have that the short exact sequence splits and the result follows.
\end{proof}

\begin{proposition}
Consider $\alpha \in \Gamma=\hilbert$ an infinite order element, and let $C$ be the infinite cyclic subgroup generated by $\alpha$. Then,
$$N_\Gamma(C)\cong \begin{cases} 

 \dbZ^{r}\oplus \dbZ/2 & \text{if $N_G(p(C))\cong \dbZ^r$, and}\\
 \dbZ^r\rtimes \dbZ/4 & \text{if $N_G(p(C))\cong \dbZ^r\rtimes \dbZ/2$.}
\end{cases}$$
\end{proposition}
\begin{proof}
First, we can assume that $A$ is either hyperbolic or parabolic up to changing the embeding of $k$ in $\dbR$. Now, from Proposition \ref{normalizersl} we obtain the following commutative diagram
$$\xymatrix{
 & Z\ar@{^{(}->}[d] \ar@{=}[r] & Z\ar@{^{(}->}[d] &  & \\
1\ar[r]& C_\Gamma(C) \ar[d]\ar[r] & N_\Gamma (C) \ar[d] \ar[r] & F \ar@{=}[d] \ar[r] & 1\\
1\ar[r]& C_G(p(C)) \ar@/^/[u]\ar[r] \ar[d] & N_G(p(C))\ar[r] \ar[d] & F\ar[r] & 1\\
 & 1 & 1 & & 
},$$
where $F$ is a subgroup of $\dbZ/2$, the middle vertical arrows are the restriction of $p: \Gamma \to G$, and the curved arrow is the splitting described in the proof of Lemma \ref{centralizergamma}. In the first case when $N_G(H)$ is a free abelian group we have that $F$ is a trivial group, then $N_\Gamma(C)\cong C_\Gamma (C)$ and the result follows from Lemma \ref{centralizergamma}. In the other case $F=\dbZ/2$ and the splitting $C_G(p(C)) \to C_\Gamma(C)$ in the diagram let us conclude that $N_\Gamma(C)$ has a finitely generated free abelian normal subgroup of index four. In order to finish the proof it is enough to show that $N_\Gamma(C)$ has an element of order four. In fact, it is well known that every finite subgroup of $SL_2(\dbR)$ is cyclic, then the preimage under $p$ of any order two subgroup of $N_G(p(C))$ is an order four cyclic subgroup of $N_\Gamma(C)$.
\end{proof}

We can finally construct a model for $\underline{\underline{E}}\Gamma$.

\begin{theorem}\label{classifyinghilbert2}
Let $\Gamma=SL_2(\calO_k)$. Let $\calI $ be defined by taking, for each $H\in I$ from Theorem \ref{classifyinghilbert}, an infinite cyclic subgroup $C$ of $\Gamma$ such that $p(C)=H$, with $p:\Gamma \to G$ the canonical projection. For every $C\in \calI$, choose models for $\underline{E}N_\Gamma(C)$ and $\underline{E}W_\Gamma(C)$, where $W_\Gamma(C)=N_\Gamma(C)/C$. Now consider the $\Gamma$-pushout:
	  $$ \xymatrix{ \coprod_{C\in \calI} \Gamma\times_{N_\Gamma(C)}\underline{E}N_\Gamma(C) \ar[r]^-i \ar[d]^{\coprod_{C\in \calI}Id_\Gamma\times_{N_\Gamma(C)}f_{C}} & \underline{E}\Gamma \ar[d] \\ \coprod_{C\in \calI}\Gamma\times_{N_\Gamma(C)} \underline{E}W_\Gamma(C) \ar[r] & X}$$
	  where $\underline{E}W_\Gamma(C)$ is viewed as an $N_\Gamma(C)$-CW-complex by restricting with the projection $N_\Gamma(C)\to W_\Gamma(C)$, the maps starting from the left upper corner are cellular and one of them is an inclusion of $\Gamma$-CW-complexes. Then $X$ is a model for $\underline{\underline{E}}\Gamma$.
\end{theorem}
\begin{proof}
Since $p:\Gamma \to G$ has finite kernel we have $p^*\underline{E}G=\underline{E}\Gamma$, and $p^*\underline{\underline{E}}G=\underline{\underline{E}}\Gamma$. On the other hand, from Lemma \ref{lemmacentranormGamma} and Lemma \ref{transformariongrouplemma} we get the identifications 
\[p^*(G\times_{N_G(p(C))} \underline{E}N_G(p(C)))=\Gamma \times_{\inv{p}(N_G(p(C)))}f^*\underline{E}N_G(p(C))=\Gamma\times_{N_\Gamma(C)}\underline{E}N_\Gamma(C),\text{ and }
\]
\[p^*(G\times_{N_G(p(C))} \underline{E}W_G(p(C)))=\Gamma \times_{\inv{p}(W_G(p(C)))}f^*\underline{E}W_G(p(C))=\Gamma\times_{W_\Gamma(C)}\underline{E}W_\Gamma(C).
\]
Now the proof is complete once we apply $p^*$ to the $G$-pushout from Theorem \ref{classifyinghilbert}.
\end{proof}


\section{Computations of algebraic K-theory}


In this section we give expresions for the Whitehead groups $Wh_q(G;R)$ and $Wh_q(\Gamma;R)$ in terms of Whitehead groups of finite groups and the nilgroups of $R$ and $R(\dbZ/2)$.

\begin{definition}
  Consider $G=\philbert$. We define the following sets:
  \begin{enumerate}
      \item $\calF$ is a collection of representatives of conjugacy classes of finite maximal subgroups of $G$.
      \item $\calP$ is a collection of representatives of conjugacy classes of maximal cyclic subgroups of $G$ generated by a totally parabolic element.
      \item $\calH_1$ (resp. $\calH_2$) is a collection of representatives of conjugacy classes of maximal cyclic subgroups of $G$ generated by a totally hyperbolic, but not hyperbolic-parabolic, element such that its normalizer is isomorphic to $\dbZ^n$ (resp. $\dbZ^n\rtimes \dbZ/2$).
      \item $\calH\calP_1$ (resp. $\calH\calP_2$) is a collection of representatives of conjugacy classes of maximal cyclic subgroups of $G$ generated by a hyperbolic-parabolic element such that its normalizer is isomorphic to $\dbZ^{n-1}$ (resp. $\dbZ^{n-1}\rtimes \dbZ/2$).
      \item $\calM_1^m$ (resp. $\calM_2^m$) is a collection of representatives of conjugacy classes of maximal cyclic subgroups of $G$ generated by a mixed element, with exactly $m$ hyperbolic components $1\leq m\leq n-1$, such that its normalizer is isomorphic to $\dbZ^m$ (resp. $\dbZ^m\rtimes \dbZ/2$).
  \end{enumerate}
\end{definition}

\begin{theorem}\label{whphilbert}
 Consider $G=\philbert$ and let $R$ be an associative ring with unitary element, then, for al $q\in\dbZ$ we have an isomorphism $$Wh_q(G;R)\cong \bigoplus\limits_{M\in \calF} Wh_q(M;R) \bigoplus N_{\calP} \bigoplus N_{\calH} \bigoplus N_{\calH\calP} \bigoplus N_{\calM},$$
 where
 \begin{align*} 
     N_{\calP} &= \bigoplus\limits_{H\in\calP}\left(\dosNil{n-1} \right) \\ 
     N_{\calH} &=  \bigoplus\limits_{H\in\calH_1}\left(\dosNil{n-1} \right) \oplus \bigoplus\limits_{H\in \calH_2}\left( \Nil{n-1}    \right)  \\
     N_{\calH\calP} &= \bigoplus\limits_{H\in\calH\calP_1}\left(\dosNil{n-2} \right) \oplus \bigoplus\limits_{H\in \calH\calP_2}\left( \Nil{n-2}    \right) \\
     N_{\calM} &= \bigoplus\limits^{n-1}_{j=1} \left(  \bigoplus\limits_{H\in\calM_1^j}\left(\dosNil{j-1} \right) \oplus \bigoplus\limits_{H\in \calM_2^j}\left( \Nil{j-1}    \right) \right).
\end{align*}
\end{theorem}
\begin{proof}
First of all, we know that $G$ satisfies the Farrell-Jones conjecture as a consequence of the main result of \cite{KLR14}. Now, we can use the splitting
$$Wh_q(G;R)\cong H_q^G(EG\to \underline{\underline{E}}G;\KR)\cong H_q^G(EG\to \underline{E}G;\KR)\oplus H_q^G(\underline{E}G\to \underline{\underline{E}}G;\KR).$$
For first term in the right hand side we have the isomorphism $$H_q^G(EG\to \underline{E}G;\KR)\cong \bigoplus\limits_{M\in \calF} Wh_q(M;R),$$ which is proved in \cite[Theorem 1.1]{BSS} for $R=\dbZ$, nevertheless, the proof carries on for general $R$ with out any changes. Now, we have to calculate $H_q^G(\underline{E}G\to \underline{\underline{E}}G;\KR)$. From Theorem \ref{classifyinghilbert} we get the isomorphism $$H_q^G(\underline{E}G\to \underline{\underline{E}}G;\KR)\cong \bigoplus\limits_{H\in I}H_q^{N_G(H)}(\underline{E}N_G(H)\to \underline{E}W_G(H);\KR),$$
where $I$ is a collection of representatives of conjugacy classes of maximal infinite virtually cyclic subgroups of $G$.\\
Clearly, the set $I$ can be expressed as the disjoint union $\calP\sqcup \calH_1 \sqcup \calH_2 \sqcup \calH\calP_1 \sqcup \calH\calP_2 \sqcup \bigsqcup\limits_{r=1}^{n-1} \calM_1^r \sqcup \bigsqcup\limits_{r=1}^{n-1} \calM_2^r.$ We define $N_{\calP}=\bigoplus\limits_{H\in \calP} H_q^{N_G(H)}(\underline{E}N_G(H)\to \underline{E}W_G(H))$, $N_{\calH}=\bigoplus\limits_{H\in \calH_1\sqcup \calH_2} H_q^{N_G(H)}(\underline{E}N_G(H)\to \underline{E}W_G(H))$, and so on.\\
From Proposition \ref{mixedcom} and Proposition \ref{parabcom} we know that for any infinite cyclic subgroup $H$ of $G$ its normalizer $N_G(H)$ is either isomorphic to $\dbZ^r$ or to $\dbZ^r\rtimes \dbZ/2$ for some $r\geq0$, and from \cite[Proof of Theorem 1.11]{LR14} we have that
$$H_q^{N_G(H)}(\underline{E}N_G(H)\to \underline{E}W_G(H);\KR)=\begin{cases}
\dosNil{r-1} & \text{if } N_G(H)\cong \dbZ^r, \text{ or }\\
\Nil{r-1}    & \text{if } N_G(H)\cong \dbZ^r\rtimes/2. 
\end{cases}$$
Now the result follows.
\end{proof}

\begin{remark}
Note that if $R$ is a regular ring, then all nil-groups $NK_i(R)$ vanish, so in this case we get $Wh_q(\philbert;R)\cong \oplus_{M\in \calF} Wh_q(M;R)$. In particular, if $R=\dbZ$, we recover Theorem 1.1 from \cite{BSS}.
\end{remark}

\begin{remark}
In \cite[Theorem C]{LPW16} they proved that if $NK_i(R)$ has finite exponent, then it is isomorphic to $\oplus_\infty F$ with $F$ a finite abelian group. Now, assuming all the nil-groups of $R$ have finite exponent, we can greatly simplify the summand $N_{\calP}\oplus N_{\calH} \oplus N_{\calP \calH} \oplus N_{\calM}$ from Theorem \ref{whphilbert}, in fact, the later would be isomorphic to $$\bigoplus\limits_{i=0}^{n-1} NK_{q-i}(R).$$
\end{remark}

\begin{remark}
If we consider a Fuchsian group $G$, this is, a discrete subgroup of $PSL_2(\dbR)$, then we can compute the commensurator of any infinite cyclic subgroup $H$ with the same argument as in Lemma \ref{mixedcom} and Lemma \ref{parabcom}. What we would obtain is that $N_G[H]$ is isomorphic to $N_G(H)$ and the later is isomorphic either to $\dbZ$ or to $\dbZ \rtimes \dbZ/2 \cong D_\infty$. Now, following the proof of Theorem \ref{whphilbert}, we get, for all $q\in\dbZ$,
\[
Wh_q(G;R)\cong \bigoplus_{M\in \calF}Wh_q(M:R) \oplus \bigoplus_{I_1} (NK_q(R)\oplus NK_q(R))\oplus \bigoplus_{I_2} NK_q(R),
\]
where $\calF$ is a complete set of representatives of conjugacy classes of maximal finite subgroups, $I_1$ (resp. $I_2$) is a complete set of representatives of conjugacy classes of maximal infinite cyclic subgroups such that $N_G(H)\cong \dbZ$ (resp. $N_G(H)\cong D_\infty$). This generalizes the main result of \cite{BJPP02} and \cite{BJPP01}. A particular interesting example is the fundamental group of an orientable closed surface of genus at least two, since these are torsion free groups the first summand vanishes, hence the Whitehead groups are just sums of copies of nil-groups.
\end{remark}

Now we are going to perform the computation of $Wh_q(\hilbert)$. In order to do so we shall need the following three lemmas.

\begin{lemma}\cite[Lemma 3.7]{LR14}\label{transformariongrouplemma}
Let $f:G_1\to G_2$ be a surjective group homomorphism. Consider a subgroup $H\leq G_2$. Let $Y$ be a $G_1$-space and $Z$ be an $H$-space. Denote by $f_H: \inv{f}(H)\to H$ the restriction of $f$. Then there is a natural $G_1$-homeomorphism
\[ 
G_1\times_{\inv{f}(H)}(Y\times f_H^*Z)\to Y\times f^*(G_2\times_HZ),
\]
where $f_H^*$ and $f^*$ are the restrictions and the actions on the products are diagonal actions.
\end{lemma}

\begin{lemma}\label{lemmaLR1}
Let $\Gamma=\hilbert$, and let $C$ be a maximal cyclic subgroup of $\Gamma$ such that $N_\Gamma(C)\cong \dbZ^r \rtimes \dbZ/4$. Then, for every $q\in \dbZ$ we have 
\begin{multline*}
H_q^{N_\Gamma(C)}( \underline{E}N_\Gamma(C) \times \underline{E}W_\Gamma(C) \to \underline{E}W_\Gamma(C);\KR) \cong \\  H_q^{N_\Gamma(C)}( \underline{E}N_\Gamma(C) \times E_{Z}W_\Gamma(C) \times E_{Z}\dbZ/4 \to E_{Z}W_\Gamma(C) \times E_{Z}\dbZ/4;\KR),
\end{multline*}
where the actions of $N_\Gamma(C)$ in $\underline{E}W_\Gamma(C)$, $E_{Z}W_\Gamma(C)$ and $E_{Z}\dbZ/4$ come from the projections $q_H: N_\Gamma(C) \to W_\Gamma(C)$ and $W_\Gamma(C)\to \dbZ/4$.
\end{lemma}
\begin{proof}
Recall that we are using the notation $G=\philbert$ and $p:\Gamma \to G$. Denote by $p_C:W_\Gamma(C) \to W_G(p(C))$ the restriction of $p$, by $q_C: N_\Gamma (C) \to W_\Gamma(C)$ the canonical projection, and by $Q_C$ the composition $p_C\circ q_C$. We have that $N_\Gamma(C)\cong \dbZ^r\rtimes \dbZ/4$, $W_\Gamma(C) \cong \dbZ^{r-1} \rtimes \dbZ/4$, $N_G(p(C))\cong \dbZ^r \rtimes \dbZ/2$, and $W_G(p(C))\cong \dbZ^{r-1}\rtimes \dbZ/2$.\\

First, we know that $W_G(p(C))$ satisfies properties $(M)$ and $(NM)$, so that we have from \cite[Corollary 2.11]{LW12} the $W_G(p(C))$-pushout
\[\LW{\coprod_{F\in J_C} W_G(p(C)) \times_F EF}{EW_G(p(C))}{\coprod_{F\in J_C} W_G(p(C))/F}{\underline{E}W_G(p(C)),}\]
where $J_C$ is a complete system of representatives of maximal finite subgroups of $W_G(p(C))$.\\

Using the fact that $W_\Gamma(C)$ is an extension of $W_G(p(C))$ by a finite group and Lemma\ref{transformariongrouplemma}, we have the identifications
\begin{align*}
p_C^*EW_G(p(C)) &= E_ZW_\Gamma(C),\\
p_C^*\underline{E}W_G(p(C)) &=\underline{E}W_\Gamma(C),\\
p_C^*W_G(p(C))\times_F EF &= W_\Gamma(C)\times_{p_C^{-q}(F)} p_C^* EF, \text{ and }\\
p_C^*W_G(p(C))/F &= W_\Gamma (C) / \inv{p_C}(F),
\end{align*}
where we are abusing of notation by writing $p_C^*EF$ instead of $p_C^*|_{\inv{p_C}(F)} EF$. Now we get the $W_\Gamma(C)$-pushout 
\[
\LW{\coprod_{F\in J_C} W_\Gamma(C)\times_{\inv{p_C}(F)}p_C^*EF }{ E_Z W_\Gamma(C) }{ \coprod_{F\in J_C} W_\Gamma(C)/\inv{p_C}(F) }{ \underline{E}W_\Gamma(C). }
\]
Consider the $W_\Gamma(C)$-homology theory that is obtained by assigning to a $W_\Gamma(C)$-CW-complex $Z$ the abelian groups $H_n^{N_\Gamma(C)}(\underline{E}N_\Gamma(C) \times q_C^*Z\to q_C^*Z;\KR)$. Using the Mayer-Vietoris sequence associated to this $W_\Gamma(C)$-homology theory, the pushout above, and the fact that $E_ZW_\Gamma(C) \times E_Z \dbZ/4$ is a model for  $E_ZW_\Gamma(C)$, we get the the desired isomorphism once we prove, for all $n\in \dbZ$ and for all $F\in J_C$, the following isomorphism
\begin{multline}\label{multline1}
H_n^{N_\Gamma(C)}(\underline{E}N_\Gamma(C) \times q_C^*( W_\Gamma(C) \times_{\inv{p_C}(F)} p_C^*EF) \to q_C^*( W_\Gamma(C) \times_{\inv{p_C}(F)} p_C^*EF);\KR)\\
\cong H_n^{N_\Gamma(C)}(\underline{E}N_\Gamma(C) \times q_H^*(W_\Gamma(C)/\inv{p_C}(F) \to q_C^*(W_\Gamma(C)/\inv{p_C}(F) ;\KR).
\end{multline}
From Lemma \ref{transformariongrouplemma} we have 
\begin{align*}
\underline{E}N_\Gamma(C)\times q_C^*(W_\Gamma(C)\times_{\inv{p_C}}p_C^*EF) &= N_\Gamma(C) \times_{\inv{Q_C}(F)} ( \underline{E} \inv{Q_C}(F) \times Q_C^* EF ),  \\
W_\Gamma(C)\times_{\inv{p_C}(F)}p_C^*EF &=  N_\Gamma(C) \times_{\inv{Q_C}(F)} Q_C^*EF ,\\
\underline{E}N_\Gamma(C)\times q_C^*(W_\Gamma(C)/ \inv{p_C}(F) ) &= N_\Gamma(C) \times_{\inv{Q_C}(F)} \underline{E} \inv{Q_C}(F), \text{ and} \\
 q_C^*(W_\Gamma(C)/ \inv{p_C}(F) &= N_\Gamma(C) \times_{\inv{Q_C}(F)} pt  .
\end{align*}
Then, the left and right hand side of (\ref{multline1}) are respectively isomorphic to
\begin{equation}\label{aux1}
H_n^{\inv{Q_C}(G)} ( \underline{E} \inv{Q_C}(F) \times Q_C^* EF \to Q_C^*EF ;\KR),
\end{equation}
and
\begin{equation}\label{aux2}
H_n^{\inv{Q_C}(G)} ( \underline{E} \inv{Q_C}(F) \to pt ;\KR),
\end{equation}
On the other hand, from Theorem \ref{EfbcV}, we obtain the $\inv{Q_C}(F)$-pushout
\[
\LW{\underline{E} \inv{Q_C}(F) \times Q_C^* EF }{ \underline{E} \inv{Q_C} (F) } {Q_C^* EF}{ E_{\fbc} \inv{Q_C}(F). }
\]
So, taking the homology of the fibers of the vertical arrows, we have
\begin{multline*}
H_n^{\inv{Q_C}(F)} ( \underline{E} \inv{Q_C}(F) \times Q_C^* EF \to Q_C^* EF  ;\KR )\\
\cong 
H_n^{\inv{Q_C}(F)} ( \underline{E} \inv{Q_C} (F) \to E_{\fbc} \inv{Q_C}(F)  ;\KR )\cong H_n^{\inv{Q_C}(F)} ( \underline{E} \inv{Q_C} (F) \to pt  ;\KR ),
\end{multline*}
where the secod isomorphisms comes from the Farrell-Jones conjecture for the family $\fbc$ of finite-by-cyclic subgroups. Now we conclude that (\ref{aux1}) and (\ref{aux2}) are isomorphic.
\end{proof}

\begin{lemma}\label{lemmaLR2}
Let $\Gamma=\hilbert$, and let $C$ be a maximal cyclic subgroup of $\Gamma$ such that $N_\Gamma(C)\cong \dbZ^r \rtimes \dbZ/4$. Then, for every $q\in \dbZ$ and every maximal finite subgroup $F$ of $\Gamma$ we have 
\begin{multline}
H_j^{N_\Gamma(C)}(\underline{E}N_\Gamma(C) \times E_ZW_\Gamma(C) \times F/Z \to E_ZW_\Gamma (C) \times F/Z ;\KR) \\
\cong \bigoplus\limits_{i=0}^{r-1}( NK_{j-i}(R(\dbZ/2)) \oplus NK_{j-i}(R(\dbZ/2)))^{\binom{r-1}{i}},
\end{multline}
and the action of $F\cong \dbZ/4$ comes from flipping each pair of of Nil-groups.
\end{lemma}
\begin{proof}
Using Lemma \ref{transformariongrouplemma}, we have

\begin{align*}
H_j^{N_\Gamma(C)}&(\underline{E}N_\Gamma(C)  \times E_ZW_\Gamma(C) \times F/Z \to E_ZW_\Gamma (C) \times F/Z ;\KR)\\ &\cong 
H_j^{N_\Gamma(C)}(N_\Gamma(C)\times_{\inv{f_C} (Z)} (\underline{E}N_\Gamma(C)\times q_C^*E_ZW_\Gamma(C))\to\underline{E}N_\Gamma(C)\times q_C^*E_ZW_\Gamma(C);\KR)\\   &\cong H_j^{\inv{f_C}(Z)}(\underline{E}N_\Gamma(C)\times q_C^*E_ZW_\Gamma(C) \to q_H^*E_ZW_\Gamma(C) ;\KR) \\ &\cong 
H_j^{\inv{f_C}(Z)}(\underline{E}\inv{f_C}(Z)\to \underline{E} \inv{f_C}(Z)/C;\KR)\\
&\cong
H_j^{\dbZ\times \dbZ/2}(T^{r-1}\times (\underline{E}(\dbZ \times \dbZ/2)\to \underline{E}\dbZ/2;\KR)\\
&\cong
\bigoplus\limits_{i=0}^{r-1}( NK_{j-i}(R(\dbZ/2)) \oplus NK_{j-i}(R(\dbZ/2)))^{\binom{r-1}{i}}
\end{align*}

Analyzing the proof of \cite[Lemma 3.9(ii)]{LR14} it is easy to see that the action of $F\cong \dbZ/4$ on this homology it is given by flipping the two copies of $NK_*(R(\dbZ/2))$ in every summand.
\end{proof}

\begin{theorem}\label{whhilbert}
 Consider $\Gamma=\hilbert$, $G=\philbert$, and let $R$ be an associative ring with unitary element, then, for all $q\in \dbZ$, we have the isomorphism
 $$Wh_q(\Gamma;R)\cong H_q^G(\underline{E}G;\KR) \bigoplus N_{\calP} \bigoplus N_{\calH} \bigoplus N_{\calH\calP} \bigoplus N_{\calM},$$
 where
 \begin{align*} 
     N_{\calP} &= \bigoplus\limits_{H\in\calP}\left(\dosNildos{n-1} \right) \\ 
     N_{\calH} &=  \bigoplus\limits_{H\in\calH_1}\left(\dosNildos{n-1} \right) \oplus \bigoplus\limits_{H\in \calH_2}\left( \Nildos{n-1}    \right)  \\
     N_{\calH\calP} &= \bigoplus\limits_{H\in\calH\calP_1}\left(\dosNildos{n-2} \right) \oplus \bigoplus\limits_{H\in \calH\calP_2}\left( \Nildos{n-2}    \right) \\
     N_{\calM} &= \bigoplus\limits^{n-1}_{j=1} \left(  \bigoplus\limits_{H\in\calM_1^j}\left(\dosNildos{j-1} \right) \oplus \bigoplus\limits_{H\in \calM_2^j}\left( \Nildos{j-1}    \right) \right). 
\end{align*}
Moreover, if $\calF$ is a complete set of representatives of conjugacy classes of maximal finite subgroups of $G$, then $H_q^G(\underline{E}G;\KR)$ fits in the long exact sequence
\[
\cdots \to \bigoplus\limits_{M\in \calF} H_q^M(EM;\KR)\to H_q^G(EG;\KR)\oplus \bigoplus\limits_{M\in \calF}K_q(R(M))\to H_q^G(\underline{E}G;\KR)\to \cdots.
\]
\end{theorem}

\begin{proof}[Proof of Theorem \ref{whhilbert}]
We know from the main result of \cite{KLR14} that $\Gamma$ satisfies the Farrell-Jones conjecture. Now from Lemma \ref{splitting} we have that 
\[Wh_q(\Gamma;R)\cong H_q^\Gamma(E\Gamma \to \underline{E}\Gamma;\KR)\oplus H_q^\Gamma(\underline{E}\Gamma \to \underline{\underline{E}}\Gamma;\KR).\]\\

Following the proof of \cite[Thoerem 1.2]{BSS} we have, for all $q\in \dbZ$ that $H_q^{\Gamma}(E\Gamma \to \underline{E}\Gamma;\KR) \cong H_q^G(\underline{E}G;\KR)$. Now, we consider the following $G$-pushout 
$$\LW{\coprod_{M\in \calF} G\times_{M}EM}{EG}{\coprod_{M\in \calF}G/M}{\underline{E}G}$$
where $\calF$ is a complete set of representatives of conjugacy classes of finite maximal subgroups of $G$, which is a consequence of $G$ satisfying the hypothesis of \cite[Corollary 2.11]{LW12}. We get the long exact sequence in the statement from the Mayer-Vietoris sequence of the $G$-pushout above.


Note that, this long exact sequence could also be obtained from the p-chain spectral sequence that converges to $H_n^{Or(G)}(*_{M\fin}, *_{Tr};\KR)$ (see \cite{BSS} and \cite[Proposition 12]{BJPP01}).\\

The next step is to compute $H_q^\Gamma(\underline{E}\Gamma \to \underline{\underline{E}}G;\KR)$. From Theorem \ref{classifyinghilbert2} we get the isomorphism $$H_q^\Gamma(\underline{E}\Gamma\to \underline{\underline{E}}\Gamma;\KR)\cong \bigoplus\limits_{C\in \calI}H_q^{N_\Gamma(C)}(\underline{E}N_\Gamma(C)\to \underline{E}W_\Gamma(C);\KR).$$
In order to compute $H_q^{N_\Gamma(C)}(\underline{E}N_\Gamma(C)\to \underline{E}W_\Gamma(C);\KR)$we are going to follow the argument in the proof of Theorem 1.11 of \cite{LR14}.
For any maximal infinite cyclic subgroup $C$ of $\Gamma$ we have
\[
N_\Gamma(C)\cong \begin{cases}
\dbZ^r\times \dbZ/2 & \text{ if } N_G(p(C))\cong \dbZ^r, \text{ and }\\
\dbZ^r\rtimes \dbZ/4 & \text{ if } N_G(p(C))\cong \dbZ^r \rtimes \dbZ/2.
\end{cases}
\]
Then we proceed by cases.\\

\textbf{Case 1:} $N_\Gamma(C)\cong \dbZ^r \times \dbZ/2$. In this case we have

\begin{align*}
H_q^{N_\Gamma(C)}(\underline{E}N_\Gamma(C) \to \underline{E}W_\Gamma (C);\KR) &\cong H_q^{\dbZ^r \times \dbZ/2}(\underline{E}(\dbZ^r\times \dbZ/2)\to \underline{E}(\dbZ^{r-1}\times \dbZ/2);\KR)\\
 &\cong H_q^{\dbZ \times \dbZ/2} (\dbT^{r-1} \times (\underline{E}(\dbZ \times \dbZ/2) \to \underline{E}\dbZ/2;\KR)\\
  &\cong \dosNildos{r-1},
\end{align*}
where $\dbT^{r-1}$ is the real $(r-1)$-dimensional torus, the second isomorphism comes from the inductive structure of the equivariant homology theory and the maximality of $C$, and the third isomorphism comes from the Atiyah-Hirzebruch spectral sequence associated to the homology theory given by sendig a $CW$- complex $Z$ to the abelian groups $H_q^{\dbZ\times \dbZ/2}(Z\times (\underline{E}(\dbZ \times \dbZ/2) \to \underline{E}\dbZ/2;\KR)$.\\

\textbf{Case 2: } $N_\Gamma(C)  \cong \dbZ^r \rtimes \dbZ/4$. Since the projection $\underline{E}N_\Gamma (C) \times q_h^*\underline{E}W_\Gamma(C)$ is a $N_\Gamma(C)$-homotopy equivalence, we have from Lemma \ref{lemmaLR1} the following
\begin{align*}
H_q^{N_\Gamma(C)}(\underline{E}N_\Gamma(C)& \to \underline q_H^*{E}W_\Gamma (C);\KR)\\ &\cong H_q^{N_\Gamma(C)}(\underline{E}N_\Gamma(C)\times q_H^*\underline{E}W_\Gamma(C) \to q_H^*\underline{E}W_\Gamma (C);\KR)\\
 &\cong H_q^{N_\Gamma(C)} ( \underline{E} N_\Gamma(C) \times q_H^*(E_ZW_\Gamma (C) \times E_Z(\dbZ/4) \to q_H^*(E_ZW_\Gamma (C) \times E_Z(\dbZ/4)) ;\KR),
\end{align*}
where the third isomorphisms comes from Lemma \ref{lemmaLR1}.\\
Next, we have the $\dbZ/4$-homology theory given by assigning to a $\dbZ/4$-CW-complex $Y$ the abelian groups 
\[
H_q^{N_\Gamma(C)}\left( \underline{E}N_\Gamma(C) \times q_H^*(E_ZW_\Gamma(C)\times Y) \to q_H^*(E_ZW_\Gamma(C)\times Y);\KR\right),
\]
so that we have an Atiyah-Hirzebruch spectral sequence converging to 
\[
H_q^{N_\Gamma(C)} ( \underline{E} N_\Gamma(C) \times q_H^*(E_ZW_\Gamma (C) \times E_Z(\dbZ/4) \to q_H^*(E_ZW_\Gamma (C) \times E_Z(\dbZ/4) ;\KR),
\]
and such that 
\[
E_{i,j}^2= H_i^{\dbZ/4}( E_Z\dbZ/4; H_j^{N_\Gamma(C)} (\underline{E}N_\Gamma(C) \times q_H^*( E_Z W_\Gamma(C) \times (\dbZ/4)/Z) \to q_H^*( E_Z W_\Gamma(C) \times (\dbZ/4)/Z);\KR).
\]
From Lemma \ref{lemmaLR2} we have
\begin{align*}
E_{i,j}^2 &\cong H_j^{\dbZ/4} (E_Z\dbZ/4;\dosNildos{r-1})\\
 &\cong H_j^{\dbZ/4}(E_Z\dbZ/4; \dbZ(\dbZ/4) \otimes _{\dbZ} (\dbZ \otimes_{\dbZ(\dbZ/4)} \dosNildos{r-1})\\
 &\cong H_j(E_Z\dbZ/4 ;\dbZ\otimes_{\dbZ(\dbZ/4)} \dosNildos{r-1})\\
 &\cong H_j(E_Z\dbZ/4 ;\Nildos{r-1})\\
 &\cong \begin{cases}
 \Nildos{r-1} &  \text{ if } i=0,\\
 0 & \text{ if } i\neq0.
 \end{cases}
\end{align*}
Therefore the proof is complete once we prove Lemma \ref{lemmaLR1} and Lemma \ref{lemmaLR2}.
\end{proof}

\begin{remark}
Note that
\[
H_i^M(EM;\mathbb{K}_\dbZ)\cong \begin{cases}
\philbert^{ab}\oplus \dbZ/2 & \text{ if } i=1,\\
\dbZ & \text{ if } i=0\text{, and}\\
0 & \text{ if } i\leq -1.
\end{cases}
\]
Therefore $H_i^M(EM;\mathbb{K}_\dbZ)\to K_i(\dbZ(M))$ it is injective for $i\leq1$ and we recover Corrollary 1.4 of \cite{BSS}.
On the other hand, in general, the so called \textit{classical assembly map} $H_i^M(EM;\KR)\to K_i(R(M))$ it is not necessarily injective, not even when $R$ is regular. For example, in \cite{UW16} Ullman and Wu show that if $R$ is a finite field of characteristic $p>2$ and $M=\dbZ/2 \times \dbZ/2$, then the classical assembly map is not injective.
\end{remark}


\section{Computations of topological K-theory of $C^*$-algebras}

In order to obtain calculations of topological K-theory of the reduced $C^*$-algebra of the Hilbert modular group, we will use a similar strategy as in \cite{BSS}. First note that, since the Hilbert modular group is a countable subgroup of $SL_2(\mathbb{R})$ ,then it satisfies the Baum-Connes conjecture as well as the reduced modular Hilbert group, see for example \cite{GHW}.

We know that $\mathbb{H}^n$ is a model for $\underbar{E}PSL_2(\mathcal{O}_k)$, and moreover from Remark 4.2 in \cite{BSS} the group $PSL_2(\mathcal{O}_k)$ satisfies (M), (NM), then by Theorem 4.1 in \cite{DL03} there is a short exact sequence

$$0\to \bigoplus_{H\in\mathcal{F}}\widetilde{K}_q^{\operatorname{top}}(C_r^*(H))\to K_q^{\operatorname{top}}(C_r^*(PSL_2(\mathcal{O}_k)))\to H_q(PSL_2(\mathcal{O}_k)\setminus\mathbb{H}^k;\mathbb{K}^{\operatorname{top}})\to0,$$
where $\mathcal{F}$ denotes a set of representatives of conjugacy classes of maximal finite subgroups of $\philbert$. 

As every element in $\mathcal{F}$ is abelian (Lemma 4.1 in \cite{BSS})  we know that $$\widetilde{K}_q^{\operatorname{top}}(C_r^*(H))=\begin{cases}\mathbb{Z}^{\sharp(H)-1}&\text{if $q$ is even}\\0&\text{if $q$ is odd},\end{cases}$$  then we have an isomorphism $$K_q^{\operatorname{top}}(C_r^*(PSL_2(\mathcal{O}_k)))\otimes\mathbb{Q}\cong\begin{cases} H_q(PSL_2(\mathcal{O}_k)\setminus\mathbb{H}^k;\dbK^{\operatorname{top}})\otimes\mathbb{Q}\oplus\mathbb{Q}^{\sharp(H)-1}&\text{if $q$ is even}\\
H_q(PSL_2(\mathcal{O}_k)\setminus\mathbb{H}^k;\dbK^{\operatorname{top}})\otimes\mathbb{Q}&\text{if $q$ is odd.}\end{cases}$$

Now, it  only remains to compute the ranks of the K-homology of the quotient $SL_2(\mathcal{O}_k)\setminus\mathbb{H}^k$. For this we use the Chern character and the knowledge of their homology groups.

\begin{theorem}\cite[Theorem 6.3]{Fr90}
 The ranks of the homology groups of $X=SL_2(\mathcal{O}_k)\setminus \mathbb{H}^k$ are given by following formulae
 \begin{enumerate}
     \item $$rk(H_0(X))=1$$
     \item $$rk(H_q(X))=0\text{, for }q\geq2k$$
     \item Assume $0<q<2k$. Then we have $$rk(H_q(X))=b^q_{univ}+b^q_{Eis}+b^m_{cusp},$$
     where
     
     \begin{enumerate}
         \item $b^q_{univ}=\begin{cases}\binom{k}{q/2}&\text{ if $q$ is even}\\
         0& \text{ if $q$ is odd,}\end{cases}$
         \item $b^q_{Eis}=\begin{cases}0& \text{if }0<q<k\\
         h\cdot \binom{k-1}{q-k}&\text{if }k\leq q<2k-1\\
         h-1&\text{if } q=2n-1,\end{cases}$
         \item $b^q_{cusp}=\begin{cases}0&\text{if }q\neq k\\\sum_{p+r=q}h^{p,r}_{cusp}&\text{if }q=k,\end{cases}$
         
     \end{enumerate}
     where $h$ is the class field number of the extension $k:\mathbb{Q}$ and $$h^{p,q}_{cusp}=\sum_{\begin{subarray}{1}b\subset\{1,\ldots,k\}\\\sharp b=p\end{subarray}}\dim[\Gamma^b,(2,\ldots,2)]_0.$$
 \end{enumerate}\end{theorem}
 As $SL_2(\mathcal{O}_k)$ is an extension of $PSL_2(\mathcal{O}_k)$ by $\mathbb{Z}/2\mathbb{Z}$, we have $$H^*(SL_2(\mathcal{O}_k)\setminus\mathbb{H}^k;\mathbb{Q})\cong H^*(PSL_2(\mathcal{O}_k)\setminus\mathbb{H}^k;\mathbb{Q}).$$Moreover, as we are taking rational coefficients we have an isomorphism$$H^*(PSL_2(\mathcal{O}_k)\setminus\mathbb{H}^k;\mathbb{Q})\cong H_*(PSL_2(\mathcal{O}_k)\setminus\mathbb{H}^k;\mathbb{Q})$$

Then by the Chern character $$H_q(PSL_2(\mathcal{O}_k)\setminus\mathbb{H}^k;\dbK^{\operatorname{top}})\otimes\mathbb{Q}\cong \bigoplus_{n\in\mathbb{Z}}H_{q+2n}(PSL_2(\mathcal{O}_k)\setminus\mathbb{H}^k;\mathbb{Q}).$$

The last isomorphism completes the calculation of the topological K-theory groups of $C_r^*(PSL_2(\mathcal{O}_k))$ with rational coefficients.
Thus we have proved the following result.
\begin{theorem}\label{topk}
 The ranks of the topological K-theory of the reduced C*-algebra $C_r^*(PSL_2(\mathcal{O}_k))$ are given by the following formulae
 $$rk(K_q^{\operatorname{top}}(C_r^*(PSL_2(\mathcal{O}_k)))=
 \begin{cases}
 \sum_{n\in\mathbb{Z}}b^{q+2n}_{univ}+b^{q+2n}_{Eis}+b^{q+2n}_{cusp}+\sum_{H\in\mathcal{F}} (|H|-1)+1&\text{if $q$ is even}\\
 \sum_{n\in\mathbb{Z}}b^{q+2n}_{univ}+b^{q+2n}_{Eis}+b^{q+2n}_{cusp}&\text{if $q$ is odd.}
 \end{cases}$$
 Where $\Gamma$ denotes the set of conjugacy classes $(H)$ of non-trivial subgroups belonging to the family
$\mathcal{MFIN}$ whose elementes are maximal finite subgroups of $G$ together with the trivial group.
\end{theorem}
\begin{remark}
A computation of the torsion part of the topological K-theory groups would depend on a complete description of the group cohomology of $\philbert$ with integral coefficients, but this problem goes beyond the scope of this paper.
\end{remark}

\bibliographystyle{alpha} 
\bibliography{myblib}
\end{document}